\documentclass[10pt]{article}
\usepackage{amsfonts}
\usepackage{amsmath}
\usepackage{color}

\newtheorem{theorem}{Theorem}

\newtheorem{claim}[theorem]{Claim}

\newtheorem{proposition}[theorem]{Proposition}

\newenvironment{proof}[1][Proof]{\noindent\textbf{#1.} }{\ \rule{0.5em}{0.5em}}
\input{tcilatex}
\begin{document}

\title{On sequences of large homoclinic solutions for a difference equations
on the integers involving oscillatory nonlinearities}
\author{Robert Stegli\'{n}ski}
\maketitle

\begin{abstract}
In this paper, we determine a concrete interval of positive parameters $%
\lambda $, for which we prove the existence of infinitely many homoclinic
solutions for a discrete problem 
\begin{equation*}
-\Delta \left( a(k)\phi _{p}(\Delta u(k-1))\right) +b(k)\phi
_{p}(u(k))=\lambda f(k,u(k)),\quad k\in 
\mathbb{Z}
,
\end{equation*}%
where the nonlinear term $f:%
\mathbb{Z}
\times 
\mathbb{R}
\rightarrow \mathbb{R}$ has an appropriate oscillatory behavior at infinity,
without any symmetry assumptions. The approach is based on critical point
theory.
\end{abstract}

\textbf{Math Subject Classifications}: 39A10, 47J30, 35B38


\section{Introduction}

\bigskip In the present paper we deal with the following nonlinear
second-order difference equation: 
\begin{equation}
\left\{ 
\begin{array}{ll}
-\Delta \left( a(k)\phi _{p}(\Delta u(k-1))\right) +b(k)\phi
_{p}(u(k))=\lambda f(k,u(k)) & \mbox{for
all $k\in\mathbb{Z}$} \\ 
u(k)\rightarrow 0 & \mbox{as $|k|\to \infty$}.%
\end{array}%
\right.  \label{eq}
\end{equation}%
Here $p>1$ is a real number, $\lambda $ is a positive real parametr, $\phi
_{p}(t)=|t|^{p-2}t$ for all $t\in {\mathbb{R}}$, $a,b:{\mathbb{Z}}%
\rightarrow \mathbb{(}0,+\infty )$, while $f:{\mathbb{Z}}\times {\mathbb{R}}%
\rightarrow {\mathbb{R}}$ is a continuous function. Moreover, the forward
difference operator is defined as $\Delta u(k-1)=u(k)-u(k-1)$. We say that a
solution $u=\{u(k)\}$ of (\ref{eq})\ is homoclinic if $\lim_{\left\vert
k\right\vert \rightarrow \infty }u(k)=0.$

\bigskip

\bigskip The problem (\ref{eq}) is in a class of partial difference
equations which usually \ describe the evolution of certain phenomena over
the course of time. The theory of nonlinear discrete dynamical systems has
been used to examine discrete models appearing in many fields such as
computing, economics, biololgy and physics.

\bigskip

Boundary value problems for difference equations can be studied in several
ways. It is well known that variational method in such problems is a
powerful tool. Many authors have applied different results of critical point
theory to prove existence and multiplicity results for the solutions of
discrete nonlinear problems. Studying such problems on bounded discrete
intervals allows for the search for solutions in a finite-dimensional Banach
space (see \cite{APR}, \cite{BC}, \cite{CIT}, \cite{CM}, \cite{BR1}). The
issue of finding solutions on unbounded intervals is more delicate. To study
such problems directly by variational methods, \cite{MG} and \cite{IT}
introduced coercive weight functions which allow for preservation of certain
compactness properties on $l^{p}$-type spaces.

\bigskip

The goal of the present paper is to establish the existence of a sequence of
homoclinic solutions for the problem (\ref{eq}), which has been studied
recently in several papers. Infinitely many solutions were obtained in \cite%
{SM} by employing Nehari manifold methods, in \cite{K} by applying a variant
of the fountain theorem (but see Section 5), and in \cite{St} by use of the
Ricceri's theorem (see \cite{BMB}, \cite{R}). In this present paper, the
result will be achieved by providing the nonlinearity with a suitable
oscillatory behavior. For this kind of nonlinearity see \cite{KMR}, \cite%
{KMT}, \cite{KRV}.

\bigskip

A special case of our contributions reads as follows. For $b:{\mathbb{Z}}%
\rightarrow {\mathbb{R}}$ and the continuous mapping $f:{\mathbb{Z}}\times {%
\mathbb{R}}\rightarrow {\mathbb{R}}$ define the following conditions:

\begin{itemize}
\item[$(B)$] $b(k)\geq b_{0}>0$ for all $k\in 
\mathbb{Z}
$, $b(k)\rightarrow +\infty $ as $\left\vert k\right\vert \rightarrow
+\infty ;$

\item[$(F_{1})$] $\displaystyle\lim_{t\rightarrow 0}\frac{\left\vert
f(k,t)\right\vert }{\left\vert t\right\vert ^{p-1}}=0$ uniformly for all $%
k\in {\mathbb{Z}}$;

\item[$(F_{2})$] $\displaystyle$there are sequences $\{c_{n}\},\{d_{n}\}$
such that $0<c_{n}<d_{n}<c_{n+1},$ $\lim_{n\rightarrow \infty }c_{n}=+\infty 
$ and $f(k,t)\leq 0$ for every $k\in 
\mathbb{Z}
$ and $t\in \lbrack c_{n},d_{n}],n\in 
\mathbb{N}
$

\item[$(F_{3})$] $\displaystyle$there is $r<0$ such that $\sup_{t\in \lbrack
r,d_{n}]}\left\vert F(\cdot .t)\right\vert \in l_{1}$ for all \ $n\in 
\mathbb{N}
;$

\item[$(F_{4}^{+})$] $\displaystyle\limsup\limits_{(k,t)\rightarrow (+\infty
,+\infty )}\frac{F(k,t)}{\left[ a(k+1)+a(k)+b(k)\right] t^{p}}=+\infty ;$

\item[$(F_{4}^{-})$] $\displaystyle\limsup\limits_{(k,t)\rightarrow (-\infty
,+\infty )}\frac{F(k,t)}{\left[ a(k+1)+a(k)+b(k)\right] t^{p}}=+\infty ;$

\item[$(F_{5})$] $\displaystyle\sup_{k\in 
\mathbb{Z}
}\left( \limsup\limits_{t\rightarrow +\infty }\frac{F(k,t)}{\left[
a(k+1)+a(k)+b(k)\right] t^{p}}\right) =+\infty ,$
\end{itemize}

where $F(k,t)$ is the primitive function of $f(k,t)$, that is $%
F(k,t)=\int_{0}^{t}f(k,s)~ds$ for every $t\in 
\mathbb{R}
$ and $k\in 
\mathbb{Z}
.$ The solutions are found in the normed space $(X,\left\Vert \cdot
\right\Vert )$, where $X=\left\{ u:{\mathbb{Z}}\rightarrow {\mathbb{R}}\ :\
\sum_{k\in {\mathbb{Z}}}\left[ a(k)\left\vert \Delta u(k-1)\right\vert
^{p}+b(k)|u(k)|^{p}\right] <\infty \right\} $ and $\Vert u\Vert =\left(
\sum_{k\in {\mathbb{Z}}}\left[ a(k)\left\vert \Delta u(k-1)\right\vert
^{p}+b(k)|u(k)|^{p}\right] \right) ^{\frac{1}{p}}.$

\begin{theorem}
\label{tw0}Assume that $(A)$, $(F_{1}),(F_{2})$ and $(F_{3})$ are satisfied.
Moreover, assume that at least one of the conditions $%
(F_{4}^{+}),(F_{4}^{-}),$ $(F_{5})$ is satisfied. Then, for any $\lambda >0,$
the problem (\ref{eq}) admits a sequence of non-negative solutions in $X$
whose norms tend to infinity.
\end{theorem}

The plan of the paper is as follows: Section 2 is devoted to our abstract
framework, while Section 3 is dedicated to the main result. In Section 4 we
give two examples of the independence of conditions $(F_{4}^{+})$ and $%
(F_{5})$. Finally, we compare our result with other known results.

\section{\protect\bigskip Abstract framework}

We begin by defining some Banach spaces. For all $1\leq p<+\infty $, we
denote $\ell ^{p}$ the set of all functions $u:{\mathbb{Z}}\rightarrow {%
\mathbb{R}}$ such that 
\begin{equation*}
\Vert u\Vert _{p}^{p}=\sum_{k\in {\mathbb{Z}}}|u(k)|^{p}<+\infty .
\end{equation*}%
Moreover, we denote $\ell ^{\infty }$ the set of all functions $u:{\mathbb{Z}%
}\rightarrow {\mathbb{R}}$ such that 
\begin{equation*}
\Vert u\Vert _{\infty }=\sup_{k\in {\mathbb{Z}}}|u(k)|<+\infty
\end{equation*}

\bigskip We set 
\begin{equation*}
X=\left\{ u:{\mathbb{Z}}\rightarrow {\mathbb{R}}\ :\ \ \sum_{k\in {\mathbb{Z}%
}}\left[ a(k)\left\vert \Delta u(k-1)\right\vert ^{p}+b(k)|u(k)|^{p}\right]
<\infty \right\}
\end{equation*}%
and 
\begin{equation*}
\Vert u\Vert =\left( \sum_{k\in {\mathbb{Z}}}\left[ a(k)\left\vert \Delta
u(k-1)\right\vert ^{p}+b(k)|u(k)|^{p}\right] \right) ^{\frac{1}{p}}.
\end{equation*}%
Clearly we have 
\begin{equation}
\Vert u\Vert _{\infty }\leq \Vert u\Vert _{p}\leq b_{0}^{-\frac{1}{p}}\Vert
u\Vert \ \mbox{for all $u\in X$.}  \label{a}
\end{equation}

\bigskip As is shown in \cite{IT}, Propositions 3, $(X,\Vert \cdot \Vert )$
is a reflexive Banach space and the embedding $X\hookrightarrow l^{p}$ is
compact.

\bigskip Let 
\begin{equation*}
\Phi (u):=\frac{1}{p}\sum_{k\in 
\mathbb{Z}
}\left[ a(k)\left\vert \Delta u(k-1)\right\vert ^{p}+b(k)\left\vert
u(k)\right\vert ^{p}\right] \ \ \ \text{for all \ \ \ }u\in X
\end{equation*}%
and 
\begin{equation*}
\Psi (u):=\sum_{k\in 
\mathbb{Z}
}F(k,u(k))\text{ \ \ for all \ \ }u\in l^{p}
\end{equation*}%
where $F(k,s)=\int_{0}^{s}f(k,t)dt$ for $s\in 
\mathbb{R}
$ and $k\in 
\mathbb{Z}
$.\ Let $J:X\rightarrow \mathbb{R}$ be the functional associated to problem (%
\ref{eq}) defined by 
\begin{equation*}
J_{\lambda }(u)=\Phi (u)-\lambda \Psi (u).
\end{equation*}

\begin{proposition}
\label{propIT}\bigskip Assume that $(A)$ and $(F_{1})$ are satisfied. Then

\begin{itemize}
\item[$(a)$] $\Phi \in C^{1}(X);$

\item[$(b)$] $\Psi \in C^{1}(l^{p})$ $\ $and \ $\Psi \in C^{1}(X)$;

\item[$(c)$] $J_{\lambda }\in C^{1}(X)$ and every critical point $u\in X$ of 
$J_{\lambda }$ is a homoclinic solution of problem (\ref{eq});

\item[$(d)$] $J_{\lambda }$ is sequentially weakly lower semicontinuous
functional\ on $X$.
\end{itemize}
\end{proposition}

This version of the lemma, parts $(a),(b)$ and $(c),$ can be proved
essentially by the same way as Propositions 5,6 and 7 in \cite{IT}, where $%
a(k)\equiv 1$ on $%
\mathbb{Z}
$ and the norm on $X$ is slightly different. See also Lemma 2.3 in \cite{K}.
The proof of part $(d)$ is standard.

\section{\protect\bigskip Main Theorem}

\bigskip Now we will formulate and prove a stronger form of Theorem \ref{tw0}%
. Let%
\begin{equation*}
B_{\pm }:=\limsup\limits_{(k,t)\rightarrow (\pm \infty ,+\infty )}\frac{%
F(k,t)}{\left[ a(k+1)+a(k)+b(k)\right] t^{p}}
\end{equation*}%
and%
\begin{equation*}
B_{0}:=\sup_{k\in 
\mathbb{Z}
}\left( \limsup\limits_{t\rightarrow +\infty }\frac{F(k,t)}{\left[
a(k+1)+a(k)+b(k)\right] t^{p}}\right) .
\end{equation*}%
Set $B=\max \{B_{\pm },B_{0}\}$. For conveniece we put $\frac{1}{+\infty }%
=0. $

\bigskip

\begin{theorem}
\label{tw1}Assume that $(A)$, $(F_{1}),(F_{2})$ and $(F_{3})$ are satisfied
and assume that $B>0$. Then, for any $\lambda >\frac{1}{Bp},$ the problem (%
\ref{eq}) admits a sequence of non-negative solutions in $X$ whose norms
tend to infinity.
\end{theorem}

\begin{proof}
Put $\lambda >\frac{1}{Bp}$ and put $\Phi ,\Psi $ and $J_{\lambda }$ as in
the previous section. By Proposition \ref{propIT} we need to find a sequence 
$\{u_{n}\}$ of critical points of $J_{\lambda }$ with non-negative terms
whose norms tend to infinity.

Let $\{c_{n}\},\{d_{n}\}$ be sequences\ and $r<0$ a number satisfying
conditions $(F_{2})$ and $(F_{3})$. For every $n\in 
\mathbb{N}
$ define the set 
\begin{equation*}
W_{n}=\left\{ u\in X:r\leq u(k)\leq d_{n}\text{ for every }k\in 
\mathbb{Z}
\right\} .
\end{equation*}

\begin{claim}
\bigskip \label{c1}For every $n\in 
\mathbb{N}
,$ the functional $J_{\lambda }$ is bounded from below on $W_{n}$ and its
infimum on $W_{n}$ is attained.
\end{claim}

\bigskip Clearly, the set $W_{n}$ is weakly closed in $X$. By condition $%
(F_{3})$ we have 
\begin{eqnarray*}
J(u) &=&\frac{1}{p}\sum_{k\in 
\mathbb{Z}
}\left[ a(k)\left\vert \Delta u(k-1)\right\vert ^{p}+b(k)\left\vert
u(k)\right\vert ^{p}\right] -\lambda \sum_{k\in 
\mathbb{Z}
}F(k,u(k)) \\
&\geq &-\lambda \sum_{k\in 
\mathbb{Z}
}\max_{t\in \lbrack r,d_{n}]}F(k,t)>-\infty
\end{eqnarray*}%
for $u\in W_{n}$. Thus, $J_{\lambda }$ is bounded from below on $W_{n}$. Let 
$\eta _{n}=\inf_{W_{n}}J_{\lambda }$ and $\{\tilde{u}_{l}\}$ be sequence in $%
X$ such that $\eta _{n}\leq J_{\lambda }(\tilde{u}_{l})\leq \eta _{n}+\frac{1%
}{l}$ for all $l\in 
\mathbb{N}
$. Then%
\begin{eqnarray*}
\frac{1}{p}\left\Vert \tilde{u}_{l}\right\Vert ^{p} &=&\frac{1}{p}\sum_{k\in 
\mathbb{Z}
}\left[ a(k)\left\vert \Delta \tilde{u}_{l}(k-1)\right\vert
^{p}+b(k)\left\vert \tilde{u}_{l}(k)\right\vert ^{p}\right] =J(\tilde{u}%
_{l})+\lambda \sum_{k\in 
\mathbb{Z}
}F(k,\tilde{u}_{l}(k)) \\
&\leq &\eta _{n}+1+\lambda \sum_{k\in 
\mathbb{Z}
}\max_{t\in \lbrack r,d_{n}]}F(k,t)
\end{eqnarray*}%
for all $l\in 
\mathbb{N}
$, i.e. $\{\tilde{u}_{l}\}$ is bounded in $X$. So, up to subsequence, $\{%
\tilde{u}_{l}\}$ weakly converges in $X$ to some $u_{n}\in W_{n}$. By the
sequentially weakly lower semicontinuity of $J_{\lambda }$ we conclude that $%
J_{\lambda }(u_{n})=\eta _{n}=\inf_{W_{n}}J_{\lambda }$. This proves Claim %
\ref{c1}.

\begin{claim}
\bigskip \label{c2}For every $n\in 
\mathbb{N}
,$ let $u_{n}\in W_{n}$ be such that $J_{\lambda
}(u_{n})=\inf_{W_{n}}J_{\lambda }$. Then, \bigskip $0\leq u_{n}(k)\leq c_{n} 
$ for all $k\in 
\mathbb{Z}
$.
\end{claim}

\bigskip Let $K=\{k\in 
\mathbb{Z}
:u_{n}(k)\notin \lbrack 0,c_{n}]\}$ and suppose that $K\neq \emptyset .$ We
then introduce the sets 
\begin{equation*}
K_{-}=\{k\in K:\ u_{n}(k)<0\}\qquad \mbox{and}\qquad K_{+}=\{k\in K:\
u_{n}(k)>c_{n}\}.
\end{equation*}%
Thus, $K=K_{-}\cup K_{+}$.

Define the truncation function $\gamma :\mathbf{R}\rightarrow \mathbf{R}$ by 
$\gamma (s)=\min (s_{+},c_{n})$, where $s_{+}=\max (s,0).$ Now, set $%
w_{n}=\gamma \circ u_{n}.$ Clearly $w_{n}\in X$. Moreover, $w_{n}(k)\in
\lbrack 0,c_{n}]$ for every $k\in 
\mathbb{Z}
$; thus $w_{n}\in W_{n}.$

We also have that $w_{n}(k)=u_{n}(k)$ for all $k\in 
\mathbb{Z}
\setminus K$, $w_{n}(k)=0$ for all $k\in K_{-}$, and $w_{n}(k)=c_{n}$ for
all $k\in K_{+}$. Furthemore, we have 
\begin{eqnarray}
J_{\lambda }(w_{n})-J_{\lambda }(u_{n}) &=&\frac{1}{p}\sum_{k\in 
\mathbb{Z}
}a(k)\left( |\Delta w_{n}(k-1)|^{p}-|\Delta u_{n}(k-1)|^{p}\right) +  \notag
\\
&&+\frac{1}{p}\sum_{k\in 
\mathbb{Z}
}b(k)\left( \left\vert w_{n}(k)\right\vert ^{p}-\left\vert
u_{n}(k)\right\vert ^{p}\right) -\lambda \sum_{k\in 
\mathbb{Z}
}[F(k,w_{n}(k))-F(k,u_{n}(k))]  \notag \\
&=&:\frac{1}{p}I_{1}+\frac{1}{p}I_{2}-\lambda I_{3}.  \label{J-J}
\end{eqnarray}%
Since $\gamma $ is a Lipschitz function with Lipschitz-constant 1, and $%
w=\gamma \circ \tilde{u}$, we have%
\begin{eqnarray}
I_{1} &=&\sum_{k\in 
\mathbb{Z}
}a(k)\left( |\Delta w_{n}(k-1)|^{p}-|\Delta u_{n}(k-1)|^{p}\right)  \notag \\
&=&\sum_{k\in 
\mathbb{Z}
}a(k)\left( |w_{n}(k)-w_{n}(k-1)|^{p}-|u_{n}(k)-u_{n}(k-1)|^{p}\right) 
\notag \\
&\leq &0.  \label{I1}
\end{eqnarray}%
Moreover, we have 
\begin{eqnarray}
I_{2} &=&\sum_{k\in 
\mathbb{Z}
}b(k)\left( \left\vert w_{n}(k)\right\vert ^{p}-\left\vert
u_{n}(k)\right\vert ^{p}\right) =\sum_{k\in K}b(k)\left( \left\vert
w_{n}(k)\right\vert ^{p}-(u_{n}(k))^{p}\right)  \notag \\
&=&\sum_{k\in K_{-}}-b(k)\left\vert u_{n}(k)\right\vert ^{p}+\sum_{k\in
K_{+}}b(k)[c_{n}^{p}-\left\vert u_{n}(k)\right\vert ^{p}] \\
&\leq &0.  \notag
\end{eqnarray}%
Next, we estimate $I_{3}$. First, $F(k,s)=0$ for $s\leq 0$, $k\in 
\mathbb{Z}
$, and consequently $\sum_{k\in K_{-}}[F(k,w_{n}(k))-F(k,u_{n}(k))]=0.$ By
the mean value theorem, for every $k\in K_{+}$, there exists $\xi _{k}\in
\lbrack c_{n},u_{n}(k)]\subset \lbrack c_{n},d_{n}]$ such that $%
F(k,w_{n}(k))-F(k,u_{n}(k))=F(k,c_{n})-F(k,u_{n}(k))=f(k,\xi
_{k})(c_{n}-u_{n}(k)).$ Taking into account hypothesis $(F_{2})$, we have
that $F(k,w_{n}(k))-F(k,u_{n}(k))\geq 0$ for every $k\in K_{+}$.
Consequently, 
\begin{eqnarray}
I_{3} &=&\sum_{k\in 
\mathbb{Z}
}[F(k,w_{n}(k))-F(k,u_{n}(k))]=\sum_{k\in K}[F(k,w_{n}(k))-F(k,u_{n}(k))] 
\notag \\
&=&\sum_{k\in K_{+}}[F(k,w_{n}(k))-F(k,u_{n}(k))]\geq 0.  \label{I3}
\end{eqnarray}%
Combining relations (\ref{I1})-(\ref{I3}) with (\ref{J-J}), we have that 
\begin{equation*}
J_{\lambda }(w_{n})-J_{\lambda }(u_{n})\leq 0.
\end{equation*}%
But $J_{\lambda }(w_{n})\geq J_{\lambda }(u_{n})=\inf_{W_{n}}J_{\lambda }$
since $w_{n}\in W_{n}$. So, every term in $J_{\lambda }(w_{n})-J_{\lambda
}(u_{n})$ should be zero. In particular, from $I_{2}$, we have 
\begin{equation*}
\sum_{k\in K_{-}}\left\vert u_{n}(k)\right\vert ^{p}=\sum_{k\in
K_{+}}[c_{n}^{p}-\left\vert u_{n}(k)\right\vert ^{p}]=0,
\end{equation*}%
which imply that $u_{n}(k)=0$ for every $k\in K_{-}$ and $u_{n}(k)=c_{n}$
for every $k\in K_{+}$. By definition of the sets $K_{-}$ and $K_{+}$, we
must have $K_{-}=K_{+}=\emptyset $, which contradicts $K_{-}\cup K_{+}=K\neq
\emptyset $; therefore $K=\emptyset $. This proves Claim \ref{c2}.

\begin{claim}
\bigskip \label{c3}For every $n\in 
\mathbb{N}
,$ let $u_{n}\in W_{n}$ be such that $J_{\lambda
}(u_{n})=\inf_{W_{n}}J_{\lambda }$. Then, $u_{n}$ is a critical point of $%
J_{\lambda }$.\bigskip
\end{claim}

It is sufficient to show that $u_{n}$ is local minimum point of $J_{\lambda
} $ in $X$. Assuming the contrary, consider a sequence $\{v_{i}\}\subset X$
which converges to $u_{n}$ and $J_{\lambda }(v_{i})<J_{\lambda
}(u_{n})=\inf_{W_{n}}J_{\lambda }$ for all $i\in 
\mathbb{N}
$. From this inequality it follows that $v_{i}\notin W_{n}$ for any $i\in 
\mathbb{N}
$. Since $v_{i}\rightarrow u_{n}$ in $X$, then due to (\ref{a}), $%
v_{i}\rightarrow u_{n}$ in $l_{\infty }$ as well. Choose a positive $\delta $
such that $\delta <\frac{1}{2}\min \{-r,d_{n}-c_{n}\}$. Then, there exists $%
i_{\delta }\in 
\mathbb{N}
$ such that $\left\Vert v_{i}-u_{n}\right\Vert _{\infty }<\delta $ for every 
$i\geq i_{\delta }$. By using Claim \ref{c2} and taking into account the
choice of the number $\delta $, we conclude that $r<v_{i}(k)<d_{n}$ for all $%
k\in 
\mathbb{Z}
$ and $i\geq i_{\delta }$, which contradicts the fact $v_{i}\notin W_{n}$.
This proves Claim \ref{c3}.

\begin{claim}
\bigskip \label{inf} For every $n\in 
\mathbb{N}
,$ let $\eta _{n}=\inf_{W_{n}}J_{\lambda }$. Then $\lim_{n\rightarrow
+\infty }\eta _{n}=-\infty $.
\end{claim}

Firstly, we assume that $B=B_{\pm }$. Without loss of generality we can
assume that $B=B_{+}$. \ We begin with $B=+\infty $. Then there exists a
number $\sigma >\frac{1}{\lambda p},$ a sequence of positive integers $%
\{k_{n}\}$ and a sequence of real numbers $\{t_{n}\}$ which tends to $%
+\infty ,$ such that%
\begin{equation*}
F(k_{n},t_{n})>\sigma (a(k_{n}+1)+a(k_{n})+b(k_{n}))t_{n}^{p}
\end{equation*}%
for all $n\in 
\mathbb{N}
$. Up to extracting a subsequence, we may assume that $d_{n}\geq t_{n}\geq 1$
for all $n\in 
\mathbb{N}
$. Define in $X$ a sequence $\{w_{n}\}$ such that, for every $n\in 
\mathbb{N}
$, $w_{n}(k_{n})=t_{n}$ and $w_{n}(k)=0$ for every $k\in 
\mathbb{Z}
\backslash \{k_{n}\}$. It is clear that $w_{n}\in W_{n}.$ One then has%
\begin{eqnarray*}
J_{\lambda }(w_{n}) &=&\frac{1}{p}\sum_{k\in 
\mathbb{Z}
}\left( a(k)\left\vert \Delta w_{n}(k-1)\right\vert ^{p}+b(k)\left\vert
w_{n}(k)\right\vert ^{p}\right) -\lambda \sum_{k\in 
\mathbb{Z}
}F(k,w_{n}(k)) \\
&<&\frac{1}{p}\left( a(k_{n}+1)+a(k_{n})\right) t_{n}^{p}+\frac{1}{p}%
b(k_{n})t_{n}^{p}-\lambda \sigma (a(k_{n}+1)+a(k_{n})+b(k_{n}))t_{n}^{p} \\
&=&\left( \frac{1}{p}-\lambda \sigma \right)
(a(k_{n}+1)+a(k_{n})+b(k_{n}))t_{n}^{p}
\end{eqnarray*}%
which gives $\lim_{n\rightarrow +\infty }J(w_{n})=-\infty $. Next, assume
that $B<+\infty $. Since $\lambda >\frac{1}{Bp}$, we can fix $\varepsilon <B-%
\frac{1}{\lambda p}$. Therefore, also taking $\{k_{n}\}$ a sequence of
positive integers and $\{t_{n}\}$ a sequence of real numbers with $%
\lim_{n\rightarrow +\infty }t_{n}=+\infty $ and $d_{n}\geq t_{n}\geq 1$ for
all $n\in 
\mathbb{N}
$ such that 
\begin{equation*}
F(k_{n},t_{n})>(B-\varepsilon )(a(k_{n}+1)+a(k_{n})+b(k_{n}))t_{n}^{p}
\end{equation*}%
for all $n\in 
\mathbb{N}
$, choosing $\{w_{n}\}$ in $W_{n}$ as above, one has 
\begin{equation*}
J_{\lambda }(w_{n})<\left( \frac{1}{p}-\lambda (B-\varepsilon )\right)
(a(k_{n}+1)+a(k_{n})+b(k_{n}))t_{n}^{p}.
\end{equation*}%
So, also in this case, $\lim_{n\rightarrow +\infty }J(w_{n})=-\infty $.

Now, assume that $B=B_{0}$. We begin with $B=+\infty $. Then there exists a
number $\sigma >\frac{1}{\lambda p}$ and an index $k_{0}\in 
\mathbb{Z}
$ such that%
\begin{equation*}
\limsup\limits_{t\rightarrow +\infty }\frac{F(k_{0},t)}{%
(a(k_{0}+1)+a(k_{0})+b(k_{0}))\left\vert t\right\vert ^{p}}>\sigma .
\end{equation*}%
\ \ Then, there exists a sequance of real numbers $\{t_{n}\}$ such that $%
\lim_{n\rightarrow +\infty }t_{n}=+\infty $ and 
\begin{equation*}
F(k_{0},t_{n})>\sigma (a(k_{0}+1)+a(k_{0})+b(k_{0}))t_{n}^{p}
\end{equation*}%
for all $n\in 
\mathbb{N}
$. Up to considering a subsequence, we may assume that $d_{n}\geq t_{n}\geq
1 $ for all $n\in 
\mathbb{N}
$. Thus, take in $X$ a sequence $\{w_{n}\}$ such that, for every $n\in 
\mathbb{N}
$, $w_{n}(k_{0})=t_{n}$ and $w_{n}(k)=0$ for every $k\in 
\mathbb{Z}
\backslash \{k_{0}\}$. Then, one has $w_{n}\in W_{n}$ and%
\begin{eqnarray*}
J_{\lambda }(w_{n}) &=&\frac{1}{p}\sum_{k\in 
\mathbb{Z}
}\left( a(k)\left\vert \Delta w_{n}(k-1)\right\vert ^{p}+b(k)\left\vert
w_{n}(k)\right\vert ^{p}\right) -\lambda \sum_{k\in 
\mathbb{Z}
}F(k,w_{n}(k)) \\
&<&\frac{1}{p}\left( a(k_{0}+1)+a(k_{0})\right) t_{n}^{p}+\frac{1}{p}%
b(k_{0})t_{n}^{p}-\lambda \sigma (a(k_{0}+1)+a(k_{0})+b(k_{0}))t_{n}^{p} \\
&=&\left( \frac{1}{p}-\lambda \sigma \right)
(a(k_{0}+1)+a(k_{0})+b(k_{0}))t_{n}^{p}
\end{eqnarray*}%
which gives $\lim_{n\rightarrow +\infty }J(w_{n})=-\infty $. Next, assume
that $B<+\infty $. Since $\lambda >\frac{1}{Bp}$, we can fix $\varepsilon >0 
$ such that $\varepsilon <B-\frac{1}{\lambda p}$. Therefore, there exists an
index $k_{0}\in 
\mathbb{Z}
$ such that%
\begin{equation*}
\limsup\limits_{t\rightarrow +\infty }\frac{F(k_{0},t)}{%
(a(k_{0}+1)+a(k_{0})+b(k_{0}))t^{p}}>B-\varepsilon .
\end{equation*}%
and taking $\{t_{n}\}$ a sequance of real numbers with $\lim_{n\rightarrow
+\infty }t_{n}=+\infty $ and $d_{n}\geq t_{n}\geq 1$ for all $n\in 
\mathbb{N}
$ and 
\begin{equation*}
F(k_{0},t_{n})>\left( B-\varepsilon \right)
(a(k_{0}+1)+a(k_{0})+b(k_{0}))t_{n}^{p}
\end{equation*}%
for all $n\in 
\mathbb{N}
$, choosing $\{w_{n}\}$ in $W_{n}$\ as above, one has%
\begin{equation*}
J_{\lambda }(w_{n})<\left( \frac{1}{p}-\lambda (B-\varepsilon )\right)
(a(k_{0}+1)+a(k_{0})+b(k_{0}))t_{n}^{p}.
\end{equation*}%
So, also in this case, $\lim_{n\rightarrow +\infty }J_{\lambda
}(w_{n})=-\infty .$ This proves Claim \ref{inf}.

Now we are ready to end the proof of Theorem \ref{tw1}. With Proposition \ref%
{propIT}, Claims \ref{c2}--\ref{inf}, up to a subsequence, we have
infinitely many pairwise distinct non-negative homoclinic solutions $u_{n}$
of (\ref{eq}) with $u_{n}\in W_{n}$. To finish the proof, we will prove that 
$\left\Vert u_{n}\right\Vert \rightarrow +\infty $ as $n\rightarrow +\infty $%
. Let us assume the contrary. Therefore, there is a subsequence $%
\{u_{n_{i}}\}$ of $\{u_{n}\}$ which is bounded in $X$. Thus, it is also
bounded in $l_{\infty }$. Consequently, we can find $m_{0}\in 
\mathbb{N}
$ such that $u_{n_{i}}\in W_{m_{0}}$ for all $i\in 
\mathbb{N}
$. Then, for every $n_{i}\geq m_{0}$ one has 
\begin{equation*}
\eta _{m_{0}}=\inf_{W_{m_{0}}}J\leq J(u_{n_{i}})=\inf_{W_{n_{i}}}J=\eta
_{n_{i}}\leq \eta _{m_{0}},
\end{equation*}%
which proves that $\eta _{n_{i}}=\eta _{m_{0}}$ for all $n_{i}\geq m_{0}$,
contradicting Claim \ref{inf}. This concludes our proof.
\end{proof}

\textbf{Remark} Theorem \ref{tw0} follows now from Theorem \ref{tw1}.

\section{\protect\bigskip Examples}

\bigskip Now, we will show the example of a function for which we can apply
Theorem \ref{tw0}. First we give an example of a function $f$ for which $%
(F_{4}^{+})$ arise, but $(F_{5})$ is not satisfied.

\textbf{Example 1} \ Let $\{a(k)\},\{b(k)\}$ be two sequences of positive
numbers such that $\lim_{k\rightarrow +\infty }b(k)=+\infty $. Let $%
\{c_{n}\},\{d_{n}\}$ be sequences\ such that $0<c_{n}<d_{n}<c_{n+1}$ and $%
\lim_{n\rightarrow \infty }c_{n}=+\infty $. Let $\{h_{n}\}$ be a sequence
such that%
\begin{equation*}
h_{n}>n\ \left( a(n+1)+a(n)+b(n)\right) c_{n+1}^{p}
\end{equation*}%
for every $n\in 
\mathbb{N}
$. For every nonpositive integer $k$ let $f(k,\cdot ):%
\mathbb{R}
\rightarrow 
\mathbb{R}
$ be identically zero function. For every positive integer $k$ let $%
f(k,\cdot ):%
\mathbb{R}
\rightarrow 
\mathbb{R}
$ be any nonnegative continuous function such that \ $f(k,t)=0$ for $t\in 
\mathbb{R}
\backslash \left( d_{k},c_{k+1}\right) $ and $%
\int_{d_{k}}^{c_{k+1}}f(k,t)dt=h_{k}$. The conditions $(F_{1})$ and $(F_{2})$%
\ are now obviously satisfied.

Set $F(k,t):=\int_{0}^{t}f(k,s)ds$ for every $t\in 
\mathbb{R}
$ and $k\in 
\mathbb{Z}
$. Since for every $n\in 
\mathbb{N}
$ and all $r<0$ only finitely many $\max_{t\in \lbrack r,d_{n}]}F(k,t)$ is
nonzero, $(F_{3})$ is satisfied. By our choosing of the sequence $\{h_{n}\}$
we have\ 
\begin{eqnarray*}
\limsup\limits_{(k,t)\rightarrow (+\infty ,+\infty )}\frac{F(k,t)}{%
(a(k+1)a(k)+b(k))\left\vert t\right\vert ^{p}} &\geq &\lim_{n\rightarrow
+\infty }\frac{F(n,c_{n+1})}{(a(n+1)+a(n)+b(n))c_{n+1}^{p}} \\
&=&\lim_{n\rightarrow +\infty }\frac{h_{n}}{(a(n+1)+a(n)+b(n))c_{n+1}^{p}}%
=+\infty
\end{eqnarray*}%
and 
\begin{equation*}
\sup_{k\in 
\mathbb{Z}
}\left( \limsup\limits_{t\rightarrow +\infty }\frac{F(k,t)}{%
(a(k+1)+a(k)+b(k))\left\vert t\right\vert ^{p}}\right) =0.
\end{equation*}

Now we give an example of a function $f$ for which $(F_{5})$ arise, but $%
(F_{4}^{+})$ is not satisfied.

\bigskip \textbf{Example 2} \ Let $\{a(k)\},\{b(k)\}$ be two sequences of
positive numbers such that $\lim_{k\rightarrow +\infty }b(k)=+\infty $. Let $%
\{c_{n}\},\{d_{n}\}$ be sequences\ such that $0<c_{n}<d_{n}<c_{n+1}$ and $%
\lim_{n\rightarrow \infty }c_{n}=+\infty $. Let $\{h_{n}\}$ be a sequence of
nonnegative numbers satisfying%
\begin{equation*}
\frac{\sum_{k=1}^{n}h_{k}}{(a(1)+a(0)+b(0))c_{n+1}^{p}}>n
\end{equation*}%
for every $n\in 
\mathbb{N}
$. Let $\tilde{f}:%
\mathbb{R}
\rightarrow 
\mathbb{R}
$ be the continuous nonnegative function given by%
\begin{equation*}
\tilde{f}(s):=\sum_{n\in 
\mathbb{N}
}\ 2h_{n}\left( c_{n+1}-d_{n}-2\left\vert s-\frac{1}{2}\left(
d_{n}+c_{n+1}\right) \right\vert \right) \cdot \mathbf{1}_{[d_{n},c_{n+1}]}
\end{equation*}%
where $\mathbf{1}_{[d,c]}$ is the indicator of the interval $[d,c].$\ We
check at once that, for every $n\in 
\mathbb{N}
,$%
\begin{equation*}
\int_{d_{n}}^{c_{n+1}}\tilde{f}(s)\ ds=h_{n}.
\end{equation*}%
Set $f(0,s):=\tilde{f}(s)$ for $s\in 
\mathbb{R}
$ and $f(k,s)=0$ for $k\in 
\mathbb{Z}
\backslash \{0\}$ and $s\in 
\mathbb{R}
$. Set $F(k,t):=\int_{0}^{t}f(k,s)ds$ for every $t\in 
\mathbb{R}
$ and $k\in 
\mathbb{Z}
$. Then $F(0,c_{n+1})=\sum_{k=1}^{n}h_{k}.$ The conditions $(F_{1}),(F_{2})$
and $(F_{3})$\ are satisied and 
\begin{eqnarray*}
\sup_{k\in 
\mathbb{Z}
}\left( \limsup\limits_{t\rightarrow +\infty }\frac{F(k,t)}{%
(a(k+1)+a(k)+b(k))\left\vert t\right\vert ^{p}}\right)
&=&\limsup\limits_{t\rightarrow +\infty }\frac{F(0,t)}{(a(1)+a(0)+b(0))\left%
\vert t\right\vert ^{p}} \\
&\geq &\lim_{n\rightarrow +\infty }\frac{F(0,c_{n+1})}{%
(a(1)+a(0)+b(0))c_{n+1}^{p}} \\
&=&\lim_{n\rightarrow +\infty }\frac{\sum_{k=1}^{n}h_{k}}{%
(a(1)+a(0)+b(0))c_{n+1}^{p}}=+\infty .
\end{eqnarray*}%
Moreover%
\begin{equation*}
\limsup\limits_{(k,t)\rightarrow (+\infty ,+\infty )}\frac{F(k,t)}{%
(a(k+1)+a(k)+b(k))t^{p}}=0.
\end{equation*}

\bigskip

\section{\protect\bigskip Comparision with other known results}

In the paper \cite{K}, the following theorem is presented:\bigskip

\begin{theorem}
Assume that a function $b:%
\mathbb{Z}
\rightarrow 
\mathbb{R}
$ and a continuous function $f:%
\mathbb{Z}
\times 
\mathbb{R}
\rightarrow 
\mathbb{R}
$ satisfy conditions:

\begin{itemize}
\item[$(B)$] $b(k)\geq b_{0}>0$ for all $k\in 
\mathbb{Z}
$, $b(k)\rightarrow +\infty $ as $\left\vert k\right\vert \rightarrow
+\infty ;$

\item[$(H_{1})$] $\displaystyle\sup_{\left\vert t\right\vert \leq
T}\left\vert F(\cdot .t)\right\vert \in l_{1}$ for all $T>0$;

\item[$(H_{2})$] $\displaystyle f(k,-t)=-f(k,t)$ for all $k\in 
\mathbb{Z}
$ and $t\in 
\mathbb{R}
;$

\item[$(H_{3})$] $\displaystyle$there exist $d>0$ and $q>p$ such that $%
~\left\vert F(k,t)\right\vert \leq d\left\vert t\right\vert ^{q}$ for all $%
k\in 
\mathbb{Z}
$ and $t\in 
\mathbb{R}
;$

\item[$(H_{4})$] $\displaystyle\lim\limits_{\left\vert t\right\vert
\rightarrow +\infty }\frac{f(k,t)t}{\left\vert t\right\vert ^{p}}=+\infty $
uniformly for all $k\in 
\mathbb{Z}
;$

\item[$(H_{5})$] $\displaystyle$there exists $\sigma \geq 1$ such that $%
\sigma \mathcal{F}(k,t)\geq \mathcal{F}(k,st)$ for $k\in 
\mathbb{Z}
,t\in 
\mathbb{R}
,$ and $s\in \lbrack 0,1],$
\end{itemize}

where $F(k,t)$ is the primitive function of $f(k,t)$, that is $%
F(k,t)=\int_{0}^{t}f(k,s)ds$ for every $t\in 
\mathbb{R}
$ and $k\in 
\mathbb{Z}
,$ and $\mathcal{F}(k,t)=tf(k,t)-pF(k,t)$. Then, for any $\lambda >0$,
problem (\ref{eq}) has a sequence $\{u_{n}(k)\}$ of nontrivial solutions
such that $J_{\lambda }(u_{n})\rightarrow +\infty $ as $n\rightarrow +\infty
.$
\end{theorem}

\bigskip As an example of function, which satisfied conditions $%
(H_{1})-(H_{5})$ is given the function 
\begin{equation*}
f(k,t)=\frac{1}{k^{\mu }}\left\vert t\right\vert ^{p-2}t\ln \left(
1+\left\vert t\right\vert ^{\nu }\right) ,\ \ \ \ \ (k,t)\in 
\mathbb{Z}
\times 
\mathbb{R}%
\end{equation*}%
with $\mu >1$ and $\nu \geq 1$. But the theorem cannot be applied to this
function, because it does not satisfy the condition $(H_{4})$. Moreover, the
conditions $(H_{1})$ and $(H_{4})$ are contradictory. Indeed, since $p>1$
the hypothesis $(H_{4})$ does give us $T_{1}>0$ such that $\left\vert
f(k,t)\right\vert \geq 1$ for all $\left\vert t\right\vert \geq T_{1}$ and $%
k\in 
\mathbb{Z}
$. Put $\alpha _{k}=F(k,T_{1})$ for all$\ k\in 
\mathbb{Z}
$. Then $\{\alpha _{k}\}\in l_{1},$ by $(H_{1})$. As $f$ is continuous we
have for $T>T_{1}$ and $k\in 
\mathbb{Z}
$%
\begin{eqnarray*}
\left\vert F(k,T)\right\vert &=&\left\vert \int_{0}^{T}f(k,t)dt\right\vert
=\left\vert \int_{0}^{T_{1}}f(k,t)dt+\int_{T_{1}}^{T}f(k,t)dt\right\vert
=\left\vert \alpha _{k}+\int_{T_{1}}^{T}f(k,t)dt\right\vert \\
&\geq &\left\vert \int_{T_{1}}^{T}f(k,t)dt\right\vert -\left\vert \alpha
_{k}\right\vert =\int_{T_{1}}^{T}\left\vert f(k,t)\right\vert dt-\left\vert
\alpha _{k}\right\vert \geq (T-T_{1})-\left\vert \alpha _{k}\right\vert ,
\end{eqnarray*}%
and so $\left\vert F(\cdot ,T)\right\vert \notin l_{1}$, contrary to $%
(H_{1}) $.

In the paper \cite{SM}, the problem (\ref{eq}) with $a(k)\equiv 1$ and $%
\lambda =1$\ was conidered. The authors obtained infinitely many pairs of
homoclinic solutions assuming, among other things, that $f(k,t)$ is odd in $%
t $ for each $k\in 
\mathbb{Z}
$, i.e. $(H_{2})$. Our Theorem \ref{tw1} has no symmetry assumptions and,
for instance, the function in our Example 1 is not odd. On the other hand,
Example 7 in \cite{SM} shows the function $f:%
\mathbb{Z}
\times 
\mathbb{R}
\rightarrow 
\mathbb{R}
$ satisfying assumptions of the main theorem in \cite{SM} with $f(k,t)>0$
for all $t>1$ and $k\in 
\mathbb{Z}
$. Such a function does not satisfy $(F_{2})$ and Theorem \ref{tw1} does not
apply to it.

In the paper \cite{St}, the problem (\ref{eq}) with $a(k)\equiv 1$ was
conidered and the following theorem was obtained.

\begin{theorem}
Assume that a function $b:%
\mathbb{Z}
\rightarrow 
\mathbb{R}
$ and a continuous function $f:%
\mathbb{Z}
\times 
\mathbb{R}
\rightarrow 
\mathbb{R}
$ satisfy conditions:

\begin{itemize}
\item[$(B)$] $b(k)\geq b_{0}>0$ for all $k\in 
\mathbb{Z}
$, $b(k)\rightarrow +\infty $ as $\left\vert k\right\vert \rightarrow
+\infty ;$

\item[$(F_{1})$] $\displaystyle\lim_{t\rightarrow 0}\frac{\left\vert
f(k,t)\right\vert }{\left\vert t\right\vert ^{p-1}}=0$ uniformly for all $%
k\in {\mathbb{Z}}$.
\end{itemize}

Put 
\begin{equation*}
A:=\liminf_{t\rightarrow +\infty }~\frac{\sum_{k\in 
\mathbb{Z}
}\max_{\left\vert \xi \right\vert \leq t}F(k,\xi )}{t^{p}},
\end{equation*}%
\begin{equation*}
B_{\pm ,\pm }:=\limsup\limits_{(k,t)\rightarrow (\pm \infty ,\pm \infty )}%
\frac{F(k,t)}{(2+b(k))\left\vert t\right\vert ^{p}},
\end{equation*}
\begin{equation*}
B_{\pm }:=\sup_{k\in 
\mathbb{Z}
}\left( \limsup\limits_{t\rightarrow \pm \infty }\frac{F(k,t)}{%
(2+b(k))\left\vert t\right\vert ^{p}}\right)
\end{equation*}%
and $B:=\max \{B_{\pm ,\pm },B_{\pm }\},$ where $F(k,t)$ is the primitive
function of $f(k,t)$. If \ $A<b_{0}\cdot B$, then for each $\lambda \in
I:=\left( \frac{1}{Bp},\frac{b_{0}}{Ap}\right) $ problem (\ref{eq}) admits a
sequence of solutions.
\end{theorem}

\bigskip As the example 3 in \cite{St} shows, for any two strictly positive
real numbers $\alpha ,\beta $ there is a continuous function $f:%
\mathbb{Z}
\times 
\mathbb{R}
\rightarrow 
\mathbb{R}
$ such that $A=\alpha $ and $B=\dot{\beta}$. So, if we choose $\alpha ,\beta
>0$ with $\alpha \geq b_{0}\cdot \beta $, we will not be able to apply the
above theorem. Since this example is similar to our Example 1, the function $%
f$ satisfies the condition $(F_{2})$ and $(F_{3})$, and we can apply Theorem %
\ref{tw1} to obtain a sequence of solutions. On the other hand, as $f$ in
example 3 in \cite{St} is non-negative, it is easy to see, that we can
modify it in the way, that for some (or even infinitlely many) $k$ we have $%
f(k,t)>0$ for all $t\geq 1$ and the interval $I$ differ by as little as we
wish. Therefore, such an $f$ does not satisfy $(F_{2})$\ and can not be used
in Theorem \ref{tw1}.

\begin{tabular}{l}
Robert Stegli\'{n}ski \\ 
Institute of Mathematics, \\ 
Lodz University of Technology, \\ 
Wolczanska 215, 90-924 Lodz, Poland, \\ 
robert.steglinski@p.lodz.pl%
\end{tabular}


\bigskip

\begin{thebibliography}{99}
\bibitem{APR} R. P. Agarwal, K. Perera and D. O'Regan, Multiple positive
solutions of singular and nonsingular discrete problems via variational
methods, \textit{Nonlinear Analysis} \textbf{58} (2004), 69-73.

\bibitem{BC} G. Bonanno, P. Candito, Infinitely many solutions for a class
of discrete non-linear boundary value problems, \textit{Appl. Anal}. \textbf{%
88} (2009), 605--616.

\bibitem{BMB} G. Bonanno, G. Molica Bisci, Infinitely many solutions for a
boundary value problem with discontinuous nonlinearities, \textit{Bound.
Value Probl}., \textbf{2009} (2009), 1--20.

\bibitem{BMB2} G. Bonanno, G. Molica Bisci, Infinitely many solutions for a
Dirichlet problem involving the $p-$Laplacian, \textit{Proceedings of the
Royal Society of Edinburgh}., \textbf{140A} (2010), 737--752.

\bibitem{CIT} A. Cabada, A. Iannizzotto, S. Tersian, Multiple solutions for
discrete boundary value problems, \textit{J. Math. Anal. Appl}. \textbf{356}
(2009), 418--428.

\bibitem{CM} P. Candito, G. Molica Bisci, Existence of two solutions for a
second-order discrete boundary value problem, \textit{Adv. Nonlinear Studies,%
} \textbf{11} (2011), 443-453.

\bibitem{IT} A. Iannizzotto, S. Tersian, Multiple homoclinic solutions for
the discrete $p-$Laplacian via critical point theory, \textit{J. Math. Anal.
Appl}. \textbf{403} (2013), 173--182.

\bibitem{K} L. Kong, Homoclinic solutions for a second order difference
equation with $p-$Laplacian, \textit{Appl. Math. Comput}., \textbf{247}
(2014), 1113--1121.

\bibitem{KMR} A. Krist\'{a}ly, M. Mih\u{a}ilescu, V. R\u{a}dulescu, Discrete
boundary value problems involving oscillatory nonlinearities: small and
large solutions, \textit{Journal of Difference Equations and Applications} 
\textbf{17} (2011), 1431-1440

\bibitem{KMT} A. Krist\'{a}ly, G. Morosanu, S. Tersian, Quasilinear elliptic
problems in $%
\mathbb{R}
^{n}$ involving oscillatory nonlinearities, \textit{J. Differential
Equations }\textbf{235} (2007), 366-375.

\bibitem{KRV} A. Krist\'{a}ly, V. R\v{a}dulescu, C. Varga, Variational
principles in mathematical physics, geometry, and economics. \textit{%
Encyclopedia of Mathematics and its Applications}, \textbf{136}. Cambridge
University Press, Cambridge, 2010.

\bibitem{MG} M. Ma, Z. Guo, Homoclinic orbits for second order self-adjont
difference equations, \textit{J. Math. Anal. Appl}. \textbf{323} (2005),
513--521.

\bibitem{BR1} G. Molica Bisci, D. Repov\v{s}, Existence of solutions for
p-Laplacian discrete equations, \textit{Appl. Math. Comput}., \textbf{242}
(2014), 454-461.

\bibitem{R} B. Ricceri, A general variational principle and some of its
applications, \textit{J. Comput. Appl. Math.} \textbf{133} (2000), 401-410.

\bibitem{St} R. Stegli\'{n}ski, On sequences of large homoclinic solutions
for a difference equations on the integers, \textit{Adv. Difference Equ.},
(2016), 2016:38.

\bibitem{SM} G. Sun, A. Mai, Infinitely many homoclinic solutions for second
order nonlinear difference equations with $p-$Laplacian, \textit{The
Scientific World Journal}, (2014).
\end{thebibliography}
\end{document}